\documentclass[pdflatex,sn-mathphys-num]{sn-jnl}

\usepackage{amsmath,amssymb,amsthm,mathtools}
\usepackage{graphicx}%
\usepackage{multirow}%
\usepackage{amsmath,amssymb,amsfonts}%
\usepackage{amsthm}%
\usepackage{mathrsfs}%
\usepackage[title]{appendix}%
\usepackage{xcolor}%
\usepackage{textcomp}%
\usepackage{manyfoot}%
\usepackage{booktabs}%
\usepackage{algorithm}%
\usepackage{algorithmicx}%
\usepackage{algpseudocode}%
\usepackage{listings}%
\usepackage{xcolor}
\usepackage{graphicx}
\usepackage{booktabs}
\usepackage{hyperref}
\hypersetup{
  colorlinks=true,
  linkcolor=blue!55!black,
  citecolor=blue!55!black,
  urlcolor=blue!55!black
}

\theoremstyle{plain}
\newtheorem{theorem}{Theorem}[section]
\newtheorem{proposition}[theorem]{Proposition}
\newtheorem{lemma}[theorem]{Lemma}

\theoremstyle{definition}
\newtheorem{definition}[theorem]{Definition}
\newtheorem{assumption}[theorem]{Assumption}
\newtheorem{remark}[theorem]{Remark}

\newcommand{\R}{\mathbb{R}}
\newcommand{\interior}[1]{\mathring{#1}}
\DeclareMathOperator{\sgn}{sgn}

\begin{document}
\title[Reflected optimal stopping with max-type rewards]{Killed resolvents and measure-valued stopping gains for reflected optimal stopping with max-type rewards}

\author*[1]{\fnm{Louis Shuo} \sur{Wang}}
\email{wang.s41@northeastern.edu}
\equalcont{These authors contributed equally to this work.}

\author[2]{\fnm{Jiguang} \sur{Yu}}
\email{jyu678@bu.edu}
\equalcont{These authors contributed equally to this work.}

\author*[3]{\fnm{Ye} \sur{Liang}}\email{ye-liang@uiowa.edu}
\equalcont{These authors contributed equally to this work.}

\affil[1]{\orgdiv{Department of Mathematics},
\orgname{Northeastern University},
\orgaddress{
\city{Boston},
\state{MA},
\postcode{02115},
\country{USA}
}}

\affil[2]{\orgdiv{College of Engineering},
\orgname{Boston University},
\orgaddress{
\city{Boston},
\state{MA},
\postcode{02215},
\country{USA}
}}

\affil[3]{\orgdiv{College of Engineering},
\orgname{The University of Iowa},
\orgaddress{
\city{Iowa City}, 
\postcode{52242}, 
\state{IA},
\country{USA}}}
  
\abstract{
We study an infinite-horizon optimal stopping problem for a normally reflected two-dimensional diffusion in the positive quadrant with nonsmooth max-type reward
\(G(x_1,x_2)=x_1\vee \alpha x_2\). The paper develops a conditional
measure-theoretic framework for the associated reflected obstacle problem. The main innovation is to show that the stopping gain
\(\Gamma=c+rG-\mathcal LG\) is a signed measure, not a function: the kink of \(G\) generates an explicit negative surface measure on
\(\Delta=\{x_1=\alpha x_2\}\). We then prove that the correct potential representation uses the resolvent of the reflected diffusion killed on first entry into the stopping set, rather than the unrestricted reflected resolvent. Under explicit monotonicity, regularity, and measure-superharmonicity assumptions, we derive an epigraph representation, a continuation-side boundary-trace condition, and a candidate verification theorem. The framework clarifies hidden regularity and uniqueness assumptions in multidimensional nonsmooth optimal stopping.
}
\keywords{Optimal stopping; reflected diffusion;
free boundary; viscosity solution; It\^o--Tanaka formula; local
time}
\pacs[MSC Classification]{60G40, 60J60, 60H30, 35R35,
35J86, 49L25, 31C25, 60J55}
\maketitle

\section{Introduction}
\label{sec:introduction}

\subsection{The problem}
\label{sec:intro-problem}
We consider an infinite-horizon optimal stopping problem for a normally
reflected diffusion $X=(X^1,X^2)$ living in the closed quadrant
$\R^2_+=[0,\infty)^2$. The decision maker chooses a stopping time $\tau$ to
maximize the discounted expected reward net of a running cost,
\begin{equation}
\label{eq:intro-value}
   V(x)
   =
   \sup_{\tau}
   \mathbb E_x\left[
      e^{-r\tau}G(X_\tau)
      -\int_0^\tau e^{-rs}c(X_s)\,ds
   \right],
   \qquad
   G(x_1,x_2)=x_1\vee\alpha x_2 .
\end{equation}
Two structural features distinguish \eqref{eq:intro-value} from the classical
one-dimensional theory and make it a genuinely two-dimensional, nonsmooth
problem. First, the state process is constrained to the quadrant by normal
reflection at the two coordinate axes, so the generator is supplemented by
Neumann-type boundary behaviour and the corner of the quadrant requires
separate attention. Second, the max-type reward is only piecewise affine: it
has a Lipschitz kink along the diagonal $\Delta=\{x_1=\alpha x_2\}$, across
which the second derivatives of $G$ concentrate as a surface measure. Both
features destroy the smoothness that the textbook free-boundary calculus
silently assumes, and both are routinely glossed over in heuristic
treatments.

Max-type rewards of the form $x_1\vee\alpha x_2$ are the canonical payoff in
the valuation of American options on the maximum of several assets and in
related exchange- and dual-strike contracts; the qualitative structure of the
associated exercise regions was analysed in \cite{broadie1997valuation,villeneuve1999exercise,bayer2020pricing,battauz2015real,dai2005american,detemple2014optimal,soner2025stopping,wang2025analysis,whitehead2012bias}. Reflected
dynamics in the orthant arise when the underlying state is a regulated or
constrained process, as in queueing networks, inventory and production models,
and economic models with reflecting barriers; the probabilistic foundations
go back to the Skorokhod-problem analyses of Tanaka \cite{tanaka1979stochastic}, Lions
and Sznitman \cite{lions1984stochastic}, Saisho \cite{saisho1987stochastic}, and Harrison
and Reiman \cite{harrison1981reflected}, with general nonsmooth domains and
oblique reflection treated in \cite{dupuis1993sdes,hino2021pathwise,lundstrom2019stochastic,wang2025analysis1,li2026wong,lipshutz2018directional,deblassie1996brownian,duan2019white,leimkuhler2023simplest,cai2026optimal}.

\subsection{Background and related literature}
\label{sec:intro-literature}
The one-dimensional theory of optimal stopping and free-boundary problems is
mature and is comprehensively surveyed in the monograph of Peskir and
Shiryaev \cite{peskir2006optimal}; the probabilistic backbone (Snell
envelopes, dynamic programming, and the connection to variational
inequalities) is classical \cite{el2006aspects,bensoussan2011applications,friedman1982variational,jacka2019compensator,martyr2016finite,christensen2021class,ma2021dynamic,ankirchner2019verification,liu2025bidirectional,lv2026robust,jeon2023variational}.
The analytic counterpart, the obstacle problem, has a well-developed
regularity theory, with the structure of the free boundary studied in depth
by Caffarelli \cite{caffarelli1998obstacle} and others \cite{caffarelli2008regularity,yu2026beyond,figalli2024constraint,liang2025global,focardi2022local,huang2024regularity,du2026free,aleksanyan2024quantitative,eberle2023structure,yu2026rigorous}. Viscosity solutions provide
the now-standard weak framework in which value functions are characterized
without a priori smoothness \cite{wang2026damage,crandall1992user,soner2024viscosity,cheng2025viscosity,wang2026algebraic}, and the
distributional or measure-valued reformulation of the one-dimensional problem
was developed by Lamberton and Zervos \cite{lamberton2013optimal}.

For multidimensional and time-inhomogeneous problems the situation is far more
delicate, and progress on the regularity of the free boundary has been recent
and hard-won. De Angelis \cite{de2015note} established continuity of the
free boundary for one-dimensional time-homogeneous diffusions under mild local
hypotheses; Peskir \cite{peskir2019continuity} proved continuity of the optimal stopping
boundary for two-dimensional diffusions, connecting it to the principle of
smooth fit and to local time--space calculus; De Angelis and Stabile
\cite{de2019lipschitz} obtained a purely probabilistic proof of local
Lipschitz continuity of the boundary in $[0,T]\times\R^d$, notably without
uniform ellipticity; and De Angelis and Peskir \cite{de2020global}
established global $C^1$ regularity of the value function. A consistent
message of this body of work is that boundary regularity is not free: it
requires problem-specific analytic or probabilistic input, typically beyond
the mere existence of a graph representation.

The potential-theoretic toolkit we use---continuous additive functionals,
their Revuz correspondence with smooth measures, and the It\^o--Tanaka and
occupation-density formulae---is standard in the theory of Markov processes
and Dirichlet forms \cite{wang2026breakdown,revuz2013continuous,fukushima2011dirichlet,yu2026pattern}, and the
Sobolev-space It\^o--Krylov calculus underlying the verification step is due
to Krylov \cite{krylov1987approach}. The interaction of reflection with obstacle
problems also appears in the reflected backward SDE formulation of El Karoui,
Kapoudjian, Pardoux, Peng, and Quenez \cite{el1997reflected}.

\subsection{Three recurring gaps, and what this paper does about them}
\label{sec:intro-gaps}
Treatments of problems of the type \eqref{eq:intro-value} frequently rely,
explicitly or implicitly, on three steps that are not in fact justified at the
stated level of generality. The purpose of this paper is to isolate these
steps, to correct them, and to assemble what remains into a transparent
conditional framework in which every nonelementary input is named.

\begin{enumerate}
\item[(i)]
\emph{Regularity is not a corollary of being a viscosity supersolution.}
A continuous viscosity supersolution of a linear elliptic inequality does not
by itself carry the Sobolev or measure structure needed to run an
It\^o--Krylov--Tanaka argument. We therefore separate the verification
principle (Theorem~\ref{thm:verification-measure}) from any regularity claim,
and we state the hypotheses---$W^{2,p}_{\mathrm{loc}}$ regularity in the
continuation set, a signed-measure extension of $(\mathcal L-r)u-c$, and
membership in a generalized It\^o class---as explicit assumptions rather than
as consequences.

\item[(ii)]
\emph{The stopping-gain object is a signed measure, not a function.}
Because $G$ has a kink on $\Delta$, the natural object
$\Gamma:=c+rG-\mathcal LG$ is not an ordinary function: its singular part is a
surface measure on $\Delta$ produced by the local time of the projected
process $Y=X^1-\alpha X^2$ at zero. We compute this diagonal measure
explicitly via Tanaka's formula and the occupation-density/co-area identity
(Section~\ref{sec:singular-stopping-gain}), obtaining
\[
   \Gamma^\Delta(dx)
   =-\frac{q(x)}{2\sqrt{1+\alpha^2}}\,\sigma_\Delta(dx),
   \qquad q=n^\top a\,n\ge0,
\]
which is nonpositive as a measure. A consequence, easy to overlook, is that a
literal pointwise sign condition such as ``$\Gamma\ge0$ on the stopping set''
is meaningless on the part of the stopping set that meets $\Delta$ in positive
length unless the diagonal component is separated out.

\item[(iii)]
\emph{The potential must be killed at the stopping time.}
This is the central correction. After applying the It\^o--Tanaka identity up
to the optimal stopping time $\tau_{\mathcal D}$, the value admits the
potential representation
\[
   V(x)=G(x)-\mathbb E_x\!\left[\int_0^{\tau_{\mathcal D}}e^{-rs}\,dA_s^\Gamma\right]
        =G(x)-R_r^{\mathcal C}\Gamma(x),
\]
where $R_r^{\mathcal C}$ is the resolvent of the reflected diffusion
killed upon entry into $\mathcal D$. It is generally not equal
to the unrestricted reflected resolvent $R_r^{\mathrm R}(\Gamma\mathbf 1_{\mathcal C})$,
because the reflected process, if continued past $\tau_{\mathcal D}$, may
re-enter the continuation region $\mathcal C$ and thereby accumulate spurious
post-stopping occupation. The optimal stopping representation only ``sees'' the
path up to $\tau_{\mathcal D}$, and the potential must be killed accordingly.
\end{enumerate}

\subsection{Main contributions}
\label{sec:intro-contributions}
Within the conditional framework, the paper makes the following contributions.
A verification theorem (Theorem~\ref{thm:verification-measure}) is proved under
explicit measure-superharmonicity hypotheses, with no hidden regularity
theorem. A conditional epigraph theorem
(Theorem~\ref{thm:conditional-epigraph}) derives the graph structure of the
stopping set from monotonicity of the stopping advantage $H=V-G$ alone, and
yields monotonicity of the boundary under a horizontal monotonicity
hypothesis; we explain (Remark~\ref{rem:graph-assumption-status}) why this
cannot be obtained from order preservation of the diffusion, the relevant
coupling having the wrong sign under our cost convention. The singular
stopping-gain measure is computed in closed form
(Section~\ref{sec:singular-stopping-gain}), including the reflection corner
analysis (Lemma~\ref{lem:vanishing-reflection-G}). The corrected
killed-resolvent representation is established
(Theorem~\ref{thm:killed-resolvent-representation}) and recast as a
continuation-side boundary-trace condition
(Proposition~\ref{prop:boundary-trace}). Finally, candidate boundaries are
validated by a verification argument
(Theorem~\ref{thm:candidate-verification}) rather than by appeal to a
Fredholm-type uniqueness statement that the trace condition alone does not
support.

The novelty is not a single hard theorem but the discipline of the framework:
each step that the heuristic literature performs silently is here either
proved or flagged as a model-specific hypothesis, and the two genuinely
substantive corrections---the measure-valued treatment of the diagonal kink
and the killed (rather than unrestricted) resolvent---are made precise and
quantitative. This is what makes the framework safe to instantiate in concrete
models, where the named assumptions can be checked one at a time.

\subsection{Organization}
\label{sec:intro-organization}
Section~\ref{sec:model-verification} fixes the probabilistic setting, the sign
convention for the obstacle problem, and the verification principle.
Section~\ref{sec:conditional-geometry} develops the conditional geometry of the
stopping set and the epigraph theorem. Section~\ref{sec:singular-stopping-gain}
computes the singular stopping-gain measure and the reflection corner term.
Section~\ref{sec:killed-resolvent} establishes the killed-resolvent
representation, the boundary-trace condition, and the candidate-verification
theorem.

\section{Model, obstacle problem, and verification}
\label{sec:model-verification}
This section fixes the probabilistic framework, the sign convention for the
obstacle problem, and the verification principle used throughout the paper.
The formulation is deliberately conditional: the reflected diffusion is
assumed to have the required Markov and continuity properties, and the
verification theorem is stated under explicit measure-superharmonicity
hypotheses. We do not infer Sobolev or measure regularity from viscosity
supersolution status alone.

\subsection{Reflected diffusion and standing assumptions}
\label{sec:reflected-diffusion-standing}
Let
\[
   \R^2_+:=[0,\infty)^2,
   \qquad
   \interior{\R}^2_+:=(0,\infty)^2.
\]
Let $W=(W^1,W^2)$ be a two-dimensional Brownian motion on a filtered
probability space
$(\Omega,\mathcal F,(\mathcal F_t)_{t\ge0},\mathbb P)$
satisfying the usual conditions. For each $x\in\R^2_+$, we consider the
normally reflected diffusion
\begin{equation}
\label{eq:reflected-sde}
   X_t
   =
   x
   +\int_0^t\mu(X_s)\,ds
   +\int_0^t\sigma(X_s)\,dW_s
   +L_t,
   \qquad t\ge0,
\end{equation}
where $L=(L^1,L^2)$ is the reflection process. The reflection is normal on the
coordinate axes:
\begin{equation}
\label{eq:normal-reflection}
   L^i \text{ is continuous and nondecreasing},\qquad
   L^i_0=0,\qquad
   \int_0^\infty \mathbf 1_{\{X_t^i>0\}}\,dL_t^i=0,
   \quad i=1,2.
\end{equation}
Let $a(x):=\sigma(x)\sigma(x)^\top$. The interior generator is
\[
   \mathcal L f(x)
   =
   \sum_{i=1}^2\mu_i(x)\partial_i f(x)
   +
   \frac12\sum_{i,j=1}^2a_{ij}(x)\partial_{ij}f(x),
   \qquad x\in\interior{\R}^2_+.
\]

\begin{assumption}[Standing assumptions on the reflected diffusion]
\label{ass:reflected-diffusion}
The following hypotheses hold.
\begin{enumerate}
\item[\textup{(A1)}]
The coefficients $\mu:\R^2_+\to\R^2$ and $\sigma:\R^2_+\to\R^{2\times2}$ are
locally Lipschitz and have at most linear growth.
\item[\textup{(A2)}]
The diffusion matrix $a=\sigma\sigma^\top$ is locally uniformly elliptic in
$\R^2_+$: for every compact $K\subset\R^2_+$ there exists $\lambda_K>0$ such
that $\xi^\top a(x)\xi\ge\lambda_K|\xi|^2$ for all $x\in K$ and $\xi\in\R^2$.
\item[\textup{(A3)}]
For each $x\in\R^2_+$, the reflected SDE
\eqref{eq:reflected-sde}--\eqref{eq:normal-reflection} has a unique strong
solution. The family $(X^x)_{x\in\R^2_+}$ is strong Markov.
\item[\textup{(A4)}]
The solution depends continuously on the initial condition in the following
pathwise sense: for every $T>0$, every compact $K\subset\R^2_+$, and some
$q\ge1$,
\[
   \lim_{y\to x}
   \mathbb E\left[
      \sup_{0\le t\le T}|X_t^y-X_t^x|^q
   \right]=0,
   \qquad x\in K.
\]
\item[\textup{(A5)}]
The reward and running cost satisfy $G(x_1,x_2)=x_1\vee \alpha x_2$ with
$\alpha>0$, and $c:\R^2_+\to[0,\infty)$ is continuous. Additional local
regularity of $c$ is imposed only when needed.
\item[\textup{(A6)}]
The value function introduced in \eqref{eq:value-function} below is finite and
has polynomial growth: there exist $C_V>0$ and $p\ge1$ such that
\[
   0\le V(x)\le C_V(1+|x|^p),
   \qquad x\in\R^2_+.
\]
Moreover, the payoff terms appearing below are integrable for all stopping
times under consideration.
\end{enumerate}
\end{assumption}

\begin{remark}[On reflected SDE assumptions]
\label{rem:reflected-sde-background}
Assumptions \textup{(A1)}--\textup{(A4)} are stated explicitly. They should
not be read as automatic consequences of the displayed SDE alone. Existence,
uniqueness, and stability for reflected SDEs are delicate properties of the
Skorokhod problem, the domain, the reflection field, and the coefficients
\cite{tanaka1979stochastic,lions1984stochastic,saisho1987stochastic,harrison1981reflected,dupuis1993sdes}.
For normal reflection in sufficiently regular domains, the classical framework
is provided by Lions and Sznitman \cite{lions1984stochastic}. In the present
paper these properties are part of the standing hypotheses.
\end{remark}

\begin{lemma}[A sufficient Lyapunov condition for discounted integrability]
\label{lem:lyapunov-integrability}
Suppose there exists a function $\Psi\in C^2(\R^2_+)$, $\Psi\ge1$, satisfying
the reflected Neumann condition
\[
   \partial_i\Psi=0
   \quad\text{on }\{x_i=0\},
   \qquad i=1,2,
\]
and constants $K\ge0$, $\lambda<r$, and $C_\Psi>0$ such that
\[
   \mathcal L\Psi(x)\le \lambda\Psi(x)+K,
   \qquad x\in\interior{\R}^2_+,
\]
and
\[
   G(x)^+ + c(x)\le C_\Psi\Psi(x),
   \qquad x\in\R^2_+.
\]
Then
\[
   \mathbb E_x\left[
      \int_0^\infty e^{-rs}c(X_s)\,ds
   \right]<\infty
   \qquad\text{and}\qquad
   \sup_{t\ge0}
   \mathbb E_x\big[e^{-rt}G(X_t)^+\big]<\infty.
\]
\end{lemma}

\begin{proof}
Apply It\^o's formula with reflection to $e^{-rt}\Psi(X_t)$ and localize. The
Neumann condition $\partial_i\Psi=0$ on $\{x_i=0\}$ removes the reflection
term, since $dL^i$ is carried by $\{X^i=0\}$. Writing
$\phi(t):=\mathbb E_x[e^{-rt}\Psi(X_t)]$, the drift inequality
$\mathcal L\Psi-r\Psi\le(\lambda-r)\Psi+K$ gives, after taking expectations
and using the martingale part has zero mean,
\[
   \phi'(t)\le-(r-\lambda)\phi(t)+K e^{-rt}.
\]
Multiplying by the integrating factor $e^{(r-\lambda)t}$ and integrating from
$0$ to $t$ yields
\[
   \mathbb E_x[e^{-rt}\Psi(X_t)]
   \le
   e^{-(r-\lambda)t}\Psi(x)
   +
   K\int_0^t e^{-(r-\lambda)(t-s)}\,e^{-rs}\,ds .
\]
Since $r-\lambda>0$ and $r>0$, the right-hand side is bounded uniformly in
$t\ge0$, so $\sup_{t\ge0}\mathbb E_x[e^{-rt}\Psi(X_t)]<\infty$; in particular
the discounted moments of $\Psi$ are finite. Integrating the differential
inequality against $\mathcal L\Psi-r\Psi\le-(r-\lambda)\Psi+K$ and using
Tonelli's theorem shows that
$\mathbb E_x\big[\int_0^\infty e^{-rs}\Psi(X_s)\,ds\big]<\infty$. The estimates
for $G^+$ and $c$ follow from the pointwise bound $G^++c\le C_\Psi\Psi$.
\end{proof}

\subsection{Value function and dynamic programming}
\label{sec:value-dpp}
Let $\mathcal T$ denote the set of $(\mathcal F_t)$-stopping times with values
in $[0,\infty]$. We use the convention
$e^{-r\tau}G(X_\tau):=0$ on $\{\tau=\infty\}$. The value function is
\begin{equation}
\label{eq:value-function}
   V(x)
   =
   \sup_{\tau\in\mathcal T}
   \mathbb E_x\left[
      e^{-r\tau}G(X_\tau)
      -
      \int_0^\tau e^{-rs}c(X_s)\,ds
   \right],
   \qquad x\in\R^2_+.
\end{equation}
The stopping and continuation sets are
\[
   \mathcal D:=\{x\in\R^2_+:V(x)=G(x)\},
   \qquad
   \mathcal C:=\{x\in\R^2_+:V(x)>G(x)\}.
\]

\begin{lemma}[Immediate stopping lower bound]
\label{lem:immediate-stopping}
For every $x\in\R^2_+$, $V(x)\ge G(x)$.
\end{lemma}
\begin{proof}
Choose the stopping time $\tau=0$ in \eqref{eq:value-function}.
\end{proof}

\begin{proposition}[Dynamic programming principle]
\label{prop:dpp}
Under Assumption~\ref{ass:reflected-diffusion}, for every bounded stopping
time $\theta$ and every $x\in\R^2_+$,
\begin{equation*}
\begin{aligned}
   V(x)
   =
   \sup_{\tau\in\mathcal T}
   \mathbb E_x\Bigg[
   &-\int_0^{\tau\wedge\theta}e^{-rs}c(X_s)\,ds
   +e^{-r\tau}G(X_\tau)\mathbf 1_{\{\tau\le\theta\}}
   \nonumber\\
   &\qquad
   +e^{-r\theta}V(X_\theta)\mathbf 1_{\{\tau>\theta\}}
   \Bigg].
   \end{aligned}
\end{equation*}
\end{proposition}
\begin{proof}
This is the standard dynamic programming principle for an optimal stopping
problem driven by a strong Markov process \cite{el2006aspects,peskir2006optimal,wang2025multi}.
The proof uses the strong Markov property at the bounded stopping time
$\theta$, concatenation of stopping rules after $\theta$, and the
integrability hypotheses in Assumption~\ref{ass:reflected-diffusion}. No
continuity of bounded stopping-time payoffs is inferred from the ordinary
Feller property.
\end{proof}

\begin{lemma}[Lower semicontinuity of the value]
\label{lem:value-lsc}
Assume, in addition to Assumption~\ref{ass:reflected-diffusion}, that the
uniform integrability required to pass to the limit below holds locally in the
initial state. Then $V$ is lower semicontinuous on $\R^2_+$. If $V$ is
continuous, then $\mathcal D$ is closed and $\mathcal C$ is open.
\end{lemma}
\begin{proof}
Fix $x\in\R^2_+$ and $\varepsilon>0$. Choose a bounded stopping time
$\tau\le T$ such that
\[
   V(x)
   \le
   \mathbb E_x\left[
      e^{-r\tau}G(X_\tau)
      -
      \int_0^\tau e^{-rs}c(X_s)\,ds
   \right]
   +\varepsilon.
\]
For $y$ close to $x$, realize $X^x$ and $X^y$ on the same probability space
and use the same stopping rule $\tau$. By the pathwise continuity in
Assumption~\ref{ass:reflected-diffusion}\textup{(A4)}, continuity of $G$ and
$c$, and the stated uniform integrability,
\[
   \mathbb E\left[
      e^{-r\tau}G(X_\tau^y)
      -\int_0^\tau e^{-rs}c(X_s^y)\,ds
   \right]
   \longrightarrow
   \mathbb E\left[
      e^{-r\tau}G(X_\tau^x)
      -\int_0^\tau e^{-rs}c(X_s^x)\,ds
   \right].
\]
Since the same $\tau$ is admissible for the process started from $y$,
\[
   V(y)
   \ge
   \mathbb E\left[
      e^{-r\tau}G(X_\tau^y)
      -\int_0^\tau e^{-rs}c(X_s^y)\,ds
   \right].
\]
Therefore $\liminf_{y\to x}V(y)\ge V(x)-\varepsilon$. Letting
$\varepsilon\downarrow0$ proves lower semicontinuity. If $V$ is continuous,
then $H:=V-G$ is continuous and nonnegative, so $\mathcal D=\{H=0\}$ is closed
and $\mathcal C=\{H>0\}$ is open.
\end{proof}

\subsection{Obstacle problem and viscosity convention}
\label{sec:obstacle-viscosity}
The obstacle problem is written in max form:
\begin{equation}
\label{eq:obstacle-max}
   \max\{(\mathcal L-r)V-c,\;G-V\}=0
   \quad\text{in }\interior{\R}^2_+.
\end{equation}
The normal reflection condition is
\begin{equation}
\label{eq:neumann-condition}
   \partial_i V=0
   \quad\text{on }\{x_i=0\},
   \qquad i=1,2,
\end{equation}
understood in the relaxed reflected-viscosity sense below. For a smooth test
function $\varphi$, define
\begin{equation}
\label{eq:F-operator}
   F[\varphi](x)
   :=
   \max\{(\mathcal L-r)\varphi(x)-c(x),\;G(x)-\varphi(x)\}.
\end{equation}
At a contact point, $u(x)=\varphi(x)$, so the second term in
\eqref{eq:F-operator} equals $G(x)-u(x)$.

\begin{definition}[Reflected viscosity solution]
\label{def:reflected-viscosity}
Let $u:\R^2_+\to\R$ be continuous and of polynomial growth.
\begin{enumerate}
\item[\textup{(i)}]
$u$ is a \emph{reflected viscosity subsolution} of
\eqref{eq:obstacle-max}--\eqref{eq:neumann-condition} if, whenever
$\varphi\in C^2(\R^2_+)$ and $u-\varphi$ has a local maximum at $x\in\R^2_+$,
with $u(x)=\varphi(x)$, the following hold. If $x\in\interior{\R}^2_+$, then
\[
   \max\{(\mathcal L-r)\varphi(x)-c(x),\;G(x)-u(x)\}\le0.
\]
If $x_i=0$ for some $i\in\{1,2\}$, then
\[
   \min\left\{
      \max\{(\mathcal L-r)\varphi(x)-c(x),\;G(x)-u(x)\},
      \partial_i\varphi(x)
   \right\}\le0.
\]
\item[\textup{(ii)}]
$u$ is a \emph{reflected viscosity supersolution} if, whenever
$\varphi\in C^2(\R^2_+)$ and $u-\varphi$ has a local minimum at $x\in\R^2_+$,
with $u(x)=\varphi(x)$, the following hold. If $x\in\interior{\R}^2_+$, then
\[
   \max\{(\mathcal L-r)\varphi(x)-c(x),\;G(x)-u(x)\}\ge0.
\]
If $x_i=0$ for some $i\in\{1,2\}$, then
\[
   \max\left\{
      \max\{(\mathcal L-r)\varphi(x)-c(x),\;G(x)-u(x)\},
      \partial_i\varphi(x)
   \right\}\ge0.
\]
\item[\textup{(iii)}]
$u$ is a \emph{reflected viscosity solution} if it is both a reflected
viscosity subsolution and a reflected viscosity supersolution.
\end{enumerate}
\end{definition}

\begin{remark}[Sign convention]
\label{rem:viscosity-sign}
The convention in Definition~\ref{def:reflected-viscosity} is the standard one
for an equation $F=0$ \cite{crandall1992user}: subsolutions are tested
from above and satisfy $F\le0$, while supersolutions are tested from below and
satisfy $F\ge0$. The reflected boundary operator is $B_i\varphi=\partial_i\varphi$
on $\{x_i=0\}$, corresponding to the inward normal reflection direction $+e_i$.
\end{remark}

\begin{proposition}[Viscosity characterization of the value]
\label{prop:value-viscosity}
Under the dynamic programming principle and the standing integrability
hypotheses, $V$ is a reflected viscosity solution of
\eqref{eq:obstacle-max}--\eqref{eq:neumann-condition}.
\end{proposition}
\begin{proof}
The proof is the standard viscosity argument from the dynamic programming
principle \cite{crandall1992user,peskir2006optimal,gao2022rolling}. At an interior
upper contact of $V-\varphi$, applying the DPP over a short time interval and
using It\^o's formula for $\varphi(X)$ gives
$\max\{(\mathcal L-r)\varphi(x)-c(x),\;G(x)-V(x)\}\le0$. At an interior lower
contact, the usual contradiction argument gives the reverse inequality. At a
boundary point $x_i=0$, the reflection term in It\^o's formula is
$\int_0^t \partial_i\varphi(X_s)\,dL_s^i$, and the relaxed Neumann condition is
encoded by the min--max boundary conditions in
Definition~\ref{def:reflected-viscosity}.
\end{proof}

\subsection{Verification under measure superharmonicity}
\label{sec:verification-measure-superharmonicity}
We now state the verification principle in the form used later. The theorem is
intentionally formulated with explicit regularity and measure assumptions. A
continuous viscosity supersolution of a linear elliptic inequality does not by
itself provide the Sobolev or measure structure needed for an
It\^o--Krylov--Tanaka formula \cite{krylov1987approach}.

Let $\mathcal O\subset\interior{\R}^2_+$ be open and define
$\tau_{\mathcal O^c}:=\inf\{t\ge0:X_t\notin\mathcal O\}$.

\begin{definition}[Generalized It\^o class]
\label{def:generalized-ito-class}
A continuous function $u:\R^2_+\to\R$ is said to belong to the generalized
It\^o class for $X$ if the following holds. Whenever the distribution
$(\mathcal L-r)u-c$ is represented in the interior by a signed Radon measure
$\mu_u$, there is a corresponding signed continuous additive functional
$A^{\mu_u}$ such that, after localization,
\[
   e^{-rt}u(X_t)
   =
   u(x)
   +\int_0^t e^{-rs}c(X_s)\,ds
   +\int_0^t e^{-rs}\,dA_s^{\mu_u}
   +M_t
   +\text{reflection boundary terms},
\]
where $M$ is a local martingale. If $u$ satisfies the reflected Neumann
condition in the trace sense, the reflection boundary terms vanish.
\end{definition}

\begin{theorem}[Verification under measure superharmonicity]
\label{thm:verification-measure}
Let $u:\R^2_+\to\R$ be continuous and of polynomial growth. Assume:
\begin{enumerate}
\item[\textup{(V1)}] $u\ge G$ on $\R^2_+$.
\item[\textup{(V2)}] $u=G$ on $\R^2_+\setminus\mathcal O$.
\item[\textup{(V3)}] For some $p>2$, $u\in W^{2,p}_{\mathrm{loc}}(\mathcal O)$,
and $(\mathcal L-r)u-c=0$ a.e.\ in $\mathcal O$.
\item[\textup{(V4)}] On $\interior{\R}^2_+$, the distribution
$(\mathcal L-r)u-c$ extends to a signed Radon measure $\mu_u$ satisfying
$\mu_u\le0$ in the sense of measures; equivalently, for every nonnegative
$\phi\in C_c^\infty(\interior{\R}^2_+)$,
$\langle(\mathcal L-r)u-c,\phi\rangle\le0$.
\item[\textup{(V5)}] $u$ belongs to the generalized It\^o class of
Definition~\ref{def:generalized-ito-class}, and the reflected boundary
contribution is nonpositive in the supermartingale calculation. In particular,
this holds if $u$ satisfies the normal-reflection condition $\partial_i u=0$
on $\{x_i=0\}$, $i=1,2$, in the relevant trace sense.
\item[\textup{(V6)}] The localization and uniform-integrability conditions
needed to take expectations and pass to the limit hold for $u(X)$, $G(X)$,
$c(X)$, and the additive functional $A^{\mu_u}$.
\end{enumerate}
Then $u\ge V$ on $\R^2_+$. If, in addition, $\tau_{\mathcal O^c}<\infty$
$\mathbb P_x$-a.s.\ for every $x$, and the stopped It\^o formula is exact on
$[0,\tau_{\mathcal O^c}]$, then $u=V$ on $\R^2_+$ and $\tau_{\mathcal O^c}$ is
optimal.
\end{theorem}
\begin{proof}
Let $\tau$ be a stopping time and let $(\tau_n)$ be a localizing sequence such
that all terms below are integrable. Applying the generalized
It\^o--Krylov--Tanaka formula \cite{krylov1987approach,revuz2013continuous} to
$e^{-rt}u(X_t)$ on $[0,\tau_n]$ gives
\[
   e^{-r\tau_n}u(X_{\tau_n})
   =
   u(x)
   +\int_0^{\tau_n}e^{-rs}c(X_s)\,ds
   +\int_0^{\tau_n}e^{-rs}\,dA_s^{\mu_u}
   +M_{\tau_n},
\]
after the reflection contribution has been controlled by \textup{(V5)}. Since
$\mu_u\le0$, the finite-variation term associated with $A^{\mu_u}$ is
nonpositive. Taking expectations gives
\[
   \mathbb E_x\left[
      e^{-r\tau_n}u(X_{\tau_n})
      -\int_0^{\tau_n}e^{-rs}c(X_s)\,ds
   \right]
   \le u(x).
\]
Using $u\ge G$,
\[
   \mathbb E_x\left[
      e^{-r\tau_n}G(X_{\tau_n})
      -\int_0^{\tau_n}e^{-rs}c(X_s)\,ds
   \right]
   \le u(x).
\]
Letting $n\to\infty$ by \textup{(V6)} and taking the supremum over $\tau$
yields $V(x)\le u(x)$.

Now apply the same formula on $[0,t\wedge\tau_{\mathcal O^c}]$. On
$\mathcal O$, assumption \textup{(V3)} gives $(\mathcal L-r)u-c=0$, so no
measure-superharmonic defect is accumulated before $\tau_{\mathcal O^c}$. Hence
\[
   u(x)
   =
   \mathbb E_x\left[
      e^{-r(t\wedge\tau_{\mathcal O^c})}
      u(X_{t\wedge\tau_{\mathcal O^c}})
      -\int_0^{t\wedge\tau_{\mathcal O^c}}e^{-rs}c(X_s)\,ds
   \right].
\]
Letting $t\to\infty$ and using \textup{(V2)} and \textup{(V6)} gives
\[
   u(x)
   =
   \mathbb E_x\left[
      e^{-r\tau_{\mathcal O^c}}G(X_{\tau_{\mathcal O^c}})
      -\int_0^{\tau_{\mathcal O^c}}e^{-rs}c(X_s)\,ds
   \right]
   \le V(x).
\]
Together with $V\le u$, this proves $u=V$, and $\tau_{\mathcal O^c}$ is
optimal.
\end{proof}

\begin{remark}[No hidden regularity theorem]
\label{rem:no-hidden-regularity}
Theorem~\ref{thm:verification-measure} is a verification theorem, not a
regularity theorem. The assumptions that $u\in W^{2,p}_{\mathrm{loc}}$, that
$(\mathcal L-r)u-c$ extends to a nonpositive signed measure, and that the
generalized It\^o formula applies are substantive hypotheses
\cite{krylov1987approach,bensoussan2011applications,yu2026from,friedman1982variational}. They are not deduced here
from viscosity supersolution status.
\end{remark}

\section{Conditional geometry of the stopping set}
\label{sec:conditional-geometry}
This section isolates the geometric assumptions needed to represent the
stopping set as an epigraph. The main point is that the graph structure is not
derived here from order preservation of the reflected diffusion, nor from
monotonicity of the infinitesimal stopping gain. It is derived from a direct
monotonicity assumption on the stopping advantage $H:=V-G$.

\subsection{Structural monotonicity assumptions}
\label{sec:structural-monotonicity}
Throughout this section define
$   H(x):=V(x)-G(x)$,
$   \mathcal D:=\{x\in\R^2_+:H(x)=0\}$,
and
$   \mathcal C:=\{x\in\R^2_+:H(x)>0\}$.
By Lemma~\ref{lem:immediate-stopping}, $H\ge0$. The following assumption is the
structural input for the graph theorem.

\begin{assumption}[Vertical monotonicity of the stopping advantage]
\label{ass:vertical-H}
For every $x_1\ge0$, the map $x_2\longmapsto H(x_1,x_2)$ is nonincreasing on
$[0,\infty)$.
\end{assumption}

When a monotonicity statement for the boundary itself is required, we impose
the following additional condition.

\begin{assumption}[Horizontal monotonicity of the stopping advantage]
\label{ass:horizontal-H}
For every $x_2\ge0$, the map $x_1\longmapsto H(x_1,x_2)$ is monotone on
$[0,\infty)$, with a direction independent of $x_2$. That is, either all
horizontal sections are nonincreasing, or all horizontal sections are
nondecreasing.
\end{assumption}

\begin{remark}[Logical status of the graph assumption]
\label{rem:graph-assumption-status}
The epigraph theorem below assumes monotonicity of $H=V-G$. We do not claim
that Assumption~\ref{ass:vertical-H} follows from order preservation of the
reflected SDE or from monotonicity of $\Gamma:=c+rG-\mathcal LG$. Order
preservation may be useful for proving monotonicity of $V$ itself under
compatible monotonicity of the reward and cost, but monotonicity of $V$ does
not imply monotonicity of $H=V-G$. In particular, a common attempted coupling
proof of vertical monotonicity has the wrong sign under the cost convention
used here. If $z\le y$ coordinatewise and $c$ is coordinatewise nonincreasing,
then along an order-preserving coupling $c(X_s^z)-c(X_s^y)\ge0$. Thus the cost
difference contributes with the opposite sign from the one needed to prove
that $H$ is vertically nonincreasing. Any verification of
Assumption~\ref{ass:vertical-H} is therefore a model-specific task.
\end{remark}

\begin{remark}[Optional order-preservation assumption]
\label{rem:order-preservation-optional}
If the reflected diffusion is order preserving and $c$ is coordinatewise
nonincreasing, then one may obtain monotonicity of $V$ by the usual coupling
argument. This observation is separate from Assumptions~\ref{ass:vertical-H}
and \ref{ass:horizontal-H}, and it is not used to prove the epigraph structure
below.
\end{remark}

\subsection{Conditional epigraph theorem}
\label{sec:conditional-epigraph}
For $x_1\ge0$, write the vertical section of the stopping set as
$   \mathcal D(x_1):=\{x_2\ge0:(x_1,x_2)\in\mathcal D\}$.

\begin{assumption}[Nonempty vertical stopping sections]
\label{ass:nonempty-sections}
For every $x_1\ge0$, $\mathcal D(x_1)\ne\varnothing$.
\end{assumption}

\begin{theorem}[Conditional epigraph structure]
\label{thm:conditional-epigraph}
Assume that $V$ is continuous. Suppose that Assumptions~\ref{ass:vertical-H}
and \ref{ass:nonempty-sections} hold. Define
$   b(x_1):=\inf\mathcal D(x_1)$ for $x_1\ge0$.
Then $b$ is well defined with values in $[0,\infty)$, and
$   \mathcal D
   =
   \{(x_1,x_2)\in\R^2_+:x_2\ge b(x_1)\}$.
   
If, in addition, Assumption~\ref{ass:horizontal-H} holds, then $b$ is
monotone. More precisely, $H$ nonincreasing in $x_1$ implies $b$ nonincreasing,
whereas $H$ nondecreasing in $x_1$ implies $b$ nondecreasing.
\end{theorem}
\begin{proof}
Fix $x_1\ge0$. By Assumption~\ref{ass:nonempty-sections}, $\mathcal D(x_1)$ is
nonempty, so $b(x_1)<\infty$. Let $z_2\in\mathcal D(x_1)$ and let $y_2\ge z_2$.
Since $H\ge0$ and $H(x_1,z_2)=0$, Assumption~\ref{ass:vertical-H} gives
$0\le H(x_1,y_2)\le H(x_1,z_2)=0$. Hence $H(x_1,y_2)=0$, so
$y_2\in\mathcal D(x_1)$. Therefore every vertical stopping section is an upper
interval.

Because $V$ and $G$ are continuous, $H$ is continuous, so $\mathcal D=\{H=0\}$
is closed. It follows that $\mathcal D(x_1)=[b(x_1),\infty)$, and hence
$\mathcal D=\{(x_1,x_2)\in\R^2_+:x_2\ge b(x_1)\}$.

Now assume horizontal monotonicity. If $H$ is nonincreasing in $x_1$, let
$x_1\le y_1$ and take $x_2>b(x_1)$, so that $H(x_1,x_2)=0$. By horizontal
monotonicity, $0\le H(y_1,x_2)\le H(x_1,x_2)=0$, so $x_2\in\mathcal D(y_1)$ and
$b(y_1)\le x_2$. Letting $x_2\downarrow b(x_1)$ gives $b(y_1)\le b(x_1)$, so
$b$ is nonincreasing. If $H$ is nondecreasing in $x_1$, the same argument with
the roles of $x_1$ and $y_1$ reversed gives $b(x_1)\le b(y_1)$, so $b$ is
nondecreasing.
\end{proof}

\subsection{Nonemptiness as an assumption or barrier condition}
\label{sec:nonemptiness}
Assumption~\ref{ass:nonempty-sections} is a boundary-at-infinity condition: for
every fixed $x_1$, immediate stopping eventually becomes optimal as $x_2$
increases. This property is not a local consequence of an inequality such as
$(\mathcal L-r)G-c\le -\eta$ for large $x_2$. Such an inequality may be useful
in concrete examples, but by itself it does not rule out stopping strategies
that wait for favourable excursions or exploit movement in the other
coordinate. A rigorous proof of nonemptiness requires a global argument. The
following barrier condition is one possible sufficient hypothesis.

\begin{assumption}[Optional superharmonic barrier at infinity]
\label{ass:barrier-infinity}
There exist $R>0$, a continuous function $W:\R^2_+\to\R$, and an open set
$\mathcal U_R\supseteq\{x\in\R^2_+:x_2>R\}$ such that:
\begin{enumerate}
\item[\textup{(B1)}] $W\ge G$ on $\R^2_+$.
\item[\textup{(B2)}] $W=G$ on $\{x\in\R^2_+:x_2\ge R\}$.
\item[\textup{(B3)}] $W$ satisfies the reflected Neumann condition on the
coordinate axes in the trace or reflected-viscosity sense.
\item[\textup{(B4)}] $W$ is globally $r$-superharmonic relative to the running
cost: $(\mathcal L-r)W-c\le0$ in the measure-superharmonic sense required by
Theorem~\ref{thm:verification-measure}.
\item[\textup{(B5)}] The verification hypotheses of
Theorem~\ref{thm:verification-measure} apply to $W$ with continuation region
contained in $\R^2_+\setminus\{x_2\ge R\}$.
\end{enumerate}
\end{assumption}

\begin{proposition}[Barrier implies nonempty sections]
\label{prop:barrier-nonempty}
If Assumption~\ref{ass:barrier-infinity} holds, then $V(x)=G(x)$ for all
$x\in\R^2_+$ with $x_2\ge R$. Consequently,
Assumption~\ref{ass:nonempty-sections} holds.
\end{proposition}
\begin{proof}
By Assumption~\ref{ass:barrier-infinity} and the verification theorem,
$W\ge V$. Since $W=G$ on $\{x_2\ge R\}$ and $V\ge G$, we obtain
$G(x)\le V(x)\le W(x)=G(x)$ for $x_2\ge R$. Hence $V=G$ on $\{x_2\ge R\}$.
Therefore, for every $x_1\ge0$, $[R,\infty)\subseteq\mathcal D(x_1)$, so
$\mathcal D(x_1)\ne\varnothing$.
\end{proof}

\begin{remark}
\label{rem:nonemptiness-warning}
Proposition~\ref{prop:barrier-nonempty} is only a sufficient condition. In
applications one may verify Assumption~\ref{ass:nonempty-sections} by other
global arguments. What is not used in this paper is an unsupported inference
from the local far-field inequality $(\mathcal L-r)G-c\le-\eta$ alone.
\end{remark}

\subsection{Regularity status}
\label{sec:regularity-status}
The preceding theorem gives a graph representation under monotonicity of the
stopping advantage. It does not prove continuity, local Lipschitz regularity,
or smooth fit of the free boundary. Those properties are delicate and depend
on additional hypotheses; they are therefore imposed explicitly when needed.
Let
$   \Delta:=\{(x_1,x_2)\in\R^2_+:x_1=\alpha x_2\}$
denote the diagonal kink of the reward.

\begin{assumption}[Local free-boundary regularity]
\label{ass:free-boundary-regularity}
Whenever a killed-resolvent boundary characterization or a boundary trace
statement is used, we assume that the boundary $b$ obtained in
Theorem~\ref{thm:conditional-epigraph} satisfies
$   b\in C((0,\infty))\cap W^{1,\infty}_{\mathrm{loc}}((0,\infty))$.
Moreover, for every compact interval $I\Subset(0,\infty)$,
$\inf_{x_1\in I}b(x_1)>0$.
\end{assumption}

\begin{remark}[Regularity is not automatic]
\label{rem:regularity-not-automatic}
Assumption~\ref{ass:free-boundary-regularity} is not a consequence of
Theorem~\ref{thm:conditional-epigraph}. It is a separate free-boundary
regularity hypothesis. In the optimal stopping literature, continuity,
Lipschitz regularity, and smooth fit are usually proved under additional
problem-specific assumptions. Continuity results for two-dimensional optimal
stopping boundaries and probabilistic proofs of local Lipschitz regularity
require hypotheses beyond the mere existence of an epigraph representation; see
Peskir \cite{peskir2019continuity}, De Angelis and Stabile \cite{de2019lipschitz},
De Angelis \cite{de2015note}, and De Angelis and Peskir
\cite{de2020global}.
\end{remark}

The following deterministic observation may be useful for excluding jumps when
a strict sign separation of the stopping gain is known. It is not a
probabilistic regularity theorem.

\begin{lemma}[Strict sign separation excludes jumps away from the diagonal]
\label{lem:sign-separation-no-jumps}
Assume that $\mathcal D=\{(x_1,x_2):x_2\ge b(x_1)\}$, where $b$ is monotone.
Let $U\Subset \interior{\R}^2_+\setminus\Delta$ be open, and suppose that
$\Gamma$ is continuous on $U$. Assume further that there exists $\eta>0$ such
that $\Gamma(x)\ge\eta$ for $x\in U\cap\mathcal D$ and $\Gamma(x)\le-\eta$ for
$x\in U\cap\mathcal C$. Then $b$ has no jump discontinuity whose vertical jump
segment is contained in $U$.
\end{lemma}
\begin{proof}
Assume, for definiteness, that $b$ is nondecreasing and has a right jump at
$x_0$: $b(x_0)<b(x_0+):=\lim_{x\downarrow x_0}b(x)$. Choose
$y_0\in(b(x_0),b(x_0+))$ such that $z:=(x_0,y_0)\in U$. Since $y_0>b(x_0)$, the
epigraph representation gives $z\in\mathcal D$, so $\Gamma(z)\ge\eta$. On the
other hand, for $x>x_0$ sufficiently close to $x_0$ one has $y_0<b(x)$, hence
$(x,y_0)\in\mathcal C$. By continuity of $\Gamma$ on $U$, letting
$x\downarrow x_0$ gives $\Gamma(z)\le-\eta$. This contradiction excludes such a
jump. The cases of left jumps and nonincreasing $b$ are identical after
reversing the relevant inequalities.
\end{proof}

\begin{assumption}[Weak smooth fit away from the diagonal]
\label{ass:weak-smooth-fit}
At $\mathcal H^1$-almost every differentiability point
$z=(x_1,b(x_1))\in \partial\mathcal C\cap\interior{\R}^2_+\setminus\Delta$ of
the free boundary, the continuation-side non-tangential gradient trace of $V$
exists and satisfies
\[
   \lim_{\substack{x\to z\\x\in\mathcal C\ \mathrm{n.t.}}}
   \nabla V(x)
   =
   \nabla G(z),
   \qquad\text{equivalently}\qquad
   \lim_{\substack{x\to z\\x\in\mathcal C\ \mathrm{n.t.}}}
   \nabla(V-G)(x)=0.
\]
\end{assumption}

\begin{remark}[Use of smooth fit]
\label{rem:smooth-fit-status}
Assumption~\ref{ass:weak-smooth-fit} is imposed when needed; it is not proved
from the standing assumptions. Even for multidimensional diffusions without
reflection, smooth fit depends on regularity of the boundary point and on
additional analytic or probabilistic hypotheses
\cite{peskir2019continuity,de2020global,lamberton2013optimal}. In the present
reflected setting, those hypotheses must be verified separately in any concrete
model.
\end{remark}

\section{The singular stopping-gain measure}
\label{sec:singular-stopping-gain}
The max-type reward $G(x_1,x_2)=x_1\vee \alpha x_2$ is not $C^2$ across the
diagonal $\Delta:=\{(x_1,x_2)\in\R^2_+:x_1=\alpha x_2\}$. Consequently, the
object $\Gamma:=c+rG-\mathcal LG$ must be interpreted as a signed measure
rather than as an ordinary function. This section gives the precise
decomposition of $\Gamma$ into its absolutely continuous part away from
$\Delta$ and its singular diagonal part. The sign convention is fixed
throughout:
$\Gamma=c+rG-\mathcal LG$.

\subsection{The absolutely continuous part}
\label{sec:Gamma-ac}
On the two open regions
$   \mathcal R_1:=\{x\in\R^2_+:x_1>\alpha x_2\}$
and
$   \mathcal R_2:=\{x\in\R^2_+:x_1<\alpha x_2\}$,
the reward is smooth and equals $G(x)=x_1$ on $\mathcal R_1$ and
$G(x)=\alpha x_2$ on $\mathcal R_2$. We write $G_{\mathrm{sm}}$ for this
piecewise smooth representative away from $\Delta$. The absolutely continuous
part of the stopping-gain measure is defined by
\begin{equation}
\label{eq:Gamma-ac-def}
   \Gamma^{\mathrm{ac}}(x)
   :=
   c(x)+rG(x)-\mathcal L G_{\mathrm{sm}}(x),
   \qquad x\in\R^2_+\setminus\Delta.
\end{equation}
Equivalently,
\[
   \Gamma^{\mathrm{ac}}(x)=c(x)+r x_1-\mu_1(x),
   \quad x\in\mathcal R_1,
   \qquad
   \Gamma^{\mathrm{ac}}(x)=c(x)+r\alpha x_2-\alpha\mu_2(x),
   \quad x\in\mathcal R_2.
\]
There is no second-derivative contribution in either smooth region because
$G_{\mathrm{sm}}$ is affine there.

\subsection{The explicit diagonal measure}
\label{sec:Gamma-delta}
Set
$\displaystyle    Y(x):=x_1-\alpha x_2$,
$\displaystyle    n:=\nabla Y=(1,-\alpha)$
and
$\displaystyle    |n|=\sqrt{1+\alpha^2}$.
Then
\[ 
G(x)=\frac{x_1+\alpha x_2+|Y(x)|}{2}.
\]
Let
$\displaystyle q(x):=n^\top a(x)n=a_{11}(x)-2\alpha a_{12}(x)+\alpha^2a_{22}(x)$.
For the one-dimensional continuous semimartingale
$Y_t:=Y(X_t)=X_t^1-\alpha X_t^2$, the continuous martingale quadratic variation
satisfies
$   d\langle Y\rangle_t=q(X_t)\,dt$.
Tanaka's formula \cite{revuz2013continuous,karatzas2014brownian} applied to $|Y_t|$
gives
$\displaystyle    d|Y_t|=\sgn(Y_t)\,dY_t+dL_t^0(Y)$,
where $L^0(Y)$ denotes the symmetric local time of $Y$ at zero. Hence the
singular finite-variation term in $dG(X_t)$ is
$\displaystyle    \tfrac12\,dL_t^0(Y)$.
The occupation-density formula gives
\[
   L_t^0(Y)
   =
   \lim_{\varepsilon\downarrow0}
   \frac{1}{2\varepsilon}
   \int_0^t
      \mathbf 1_{\{|Y(X_s)|<\varepsilon\}}
      q(X_s)\,ds.
\]
In distributional notation this may be written as
$\displaystyle    dL_t^0(Y)=q(X_t)\,\delta_0(Y(X_t))\,dt$.
Since
$\displaystyle    \delta_0(Y(x))\,dx
   =
   \frac{1}{|\nabla Y|}\,\sigma_\Delta(dx)
   =
   \frac{1}{\sqrt{1+\alpha^2}}\,\sigma_\Delta(dx)$,
where $\sigma_\Delta$ is one-dimensional surface measure on $\Delta$, the
singular part of $\mathcal LG$ is
$\displaystyle  (\mathcal LG)^\Delta(dx)
   =
   \frac{q(x)}{2\sqrt{1+\alpha^2}}\,\sigma_\Delta(dx)$.
Therefore, with the convention $\Gamma=c+rG-\mathcal LG$, the diagonal part of
$\Gamma$ is the signed measure
\begin{equation}
\label{eq:Gamma-delta}
   \Gamma^\Delta(dx)
   =
   -\frac{q(x)}{2\sqrt{1+\alpha^2}}\,\sigma_\Delta(dx).
\end{equation}
Equivalently, for every bounded Borel function $F$,
\[
   \int_{\R^2_+}F(x)\,\Gamma^\Delta(dx)
   =
   -\frac{1}{2\sqrt{1+\alpha^2}}
   \int_\Delta F(x)q(x)\,\sigma_\Delta(dx).
\]
Combining \eqref{eq:Gamma-ac-def} and \eqref{eq:Gamma-delta}, we obtain the
measure decomposition
\begin{equation}
\label{eq:Gamma-decomposition}
   \Gamma=\Gamma^{\mathrm{ac}}\,dx+\Gamma^\Delta.
\end{equation}

\begin{remark}[Sign of the diagonal contribution]
\label{rem:Gamma-delta-sign}
Since $a(x)$ is nonnegative definite, $q(x)=n^\top a(x)n\ge0$. Hence
$\Gamma^\Delta\le0$ as a signed measure. Therefore a condition such as
``$\Gamma\ge0$ on a stopping set'' cannot be imposed literally as a measure if
the stopping set contains a positive-length portion of $\Delta$, unless the
diagonal component is separated from the absolutely continuous part. In later
arguments, any positivity condition on $\Gamma$ must specify whether it
concerns $\Gamma^{\mathrm{ac}}$ away from $\Delta$, the full signed measure, or
a decomposition in which the diagonal part is treated separately.
\end{remark}

\subsection{Reflection and the corner term}
\label{sec:reflection-corner}
The generalized It\^o--Tanaka formula for $G(X)$ also contains possible
reflection terms. Away from the origin these terms vanish formally because
$\partial_1G=0$ on $\{x_1=0,\ x_2>0\}$ and $\partial_2G=0$ on
$\{x_2=0,\ x_1>0\}$. Indeed, on $\{x_1=0,\ x_2>0\}$ one has $G=\alpha x_2$,
while on $\{x_2=0,\ x_1>0\}$ one has $G=x_1$. However, $G$ is nonsmooth at the
corner $(0,0)$. Thus the disappearance of the reflection contribution is not
automatic unless the possible corner term is controlled. We impose the
following approximation condition.

\begin{assumption}[No corner reflection contribution]
\label{ass:no-corner-reflection}
There exists a family $G_\varepsilon\in C^2(\R^2_+)$, $\varepsilon>0$, such
that $G_\varepsilon\to G$ locally uniformly on $\R^2_+$, and, for every
compact $K\subset\{x_1=0,\ x_2>0\}$, $\partial_1G_\varepsilon\to0$ uniformly on
$K$, while for every compact $K\subset\{x_2=0,\ x_1>0\}$,
$\partial_2G_\varepsilon\to0$ uniformly on $K$. Moreover, for every $t>0$ and
every starting point $x\in\R^2_+$,
\begin{align}
\label{eq:no-corner-condition}
   \lim_{\varepsilon\downarrow0}
   \mathbb E_x\Bigg[
      \int_0^t
      |\partial_1G_\varepsilon(X_s)|
      \mathbf 1_{\{X_s=(0,0)\}}\,dL_s^1+
      \int_0^t
      |\partial_2G_\varepsilon(X_s)|
      \mathbf 1_{\{X_s=(0,0)\}}\,dL_s^2
   \Bigg]
   =0.
\end{align}
\end{assumption}

\begin{lemma}[Vanishing reflection contribution for $G$]
\label{lem:vanishing-reflection-G}
Under Assumption~\ref{ass:no-corner-reflection}, the reflection contribution in
the generalized It\^o--Tanaka formula for $G(X)$ vanishes. More precisely,
\[
   \lim_{\varepsilon\downarrow0}
   \mathbb E_x\left[
      \sum_{i=1}^2
      \int_0^t \partial_iG_\varepsilon(X_s)\,dL_s^i
   \right]
   =0,
   \qquad t>0,\quad x\in\R^2_+.
\]
Consequently, no additional corner additive functional appears in the
It\^o--Tanaka formula for $G(X)$.
\end{lemma}
\begin{proof}
By the normal-reflection condition, $dL_s^i$ is carried by $\{X_s^i=0\}$,
$i=1,2$. For $i=1$, the part of $dL_s^1$ carried by $\{X_s^1=0,\ X_s^2>0\}$
gives no contribution in the limit, because $\partial_1G_\varepsilon\to0$
locally uniformly on compact subsets of that face. For $i=2$, the same
argument applies on $\{X_s^2=0,\ X_s^1>0\}$, using
$\partial_2G_\varepsilon\to0$. The only remaining possible contribution comes
from times at which $X_s=(0,0)$. This contribution is precisely controlled by
\eqref{eq:no-corner-condition}. Hence the expected reflection term tends to
zero as $\varepsilon\downarrow0$.
\end{proof}

\begin{remark}[If the corner condition fails]
\label{rem:corner-measure-warning}
If Assumption~\ref{ass:no-corner-reflection} is not verified, then the
It\^o--Tanaka formula for $G(X)$ may contain an additional corner additive
functional supported at $(0,0)$. In that case the stopping-gain measure must be
enlarged to
$\Gamma=\Gamma^{\mathrm{ac}}\,dx+\Gamma^\Delta+\Gamma^{\mathrm{corner}}$,
where $\Gamma^{\mathrm{corner}}$ is the signed measure associated with the
corner reflection contribution. The present paper works under
Assumption~\ref{ass:no-corner-reflection} and therefore takes
$\Gamma^{\mathrm{corner}}=0$.
\end{remark}

\subsection{It\^o--Tanaka identity for the reward}
\label{sec:ito-tanaka-G-final}
Under Assumption~\ref{ass:no-corner-reflection}, the generalized It\^o--Tanaka
formula for $G(X)$ may be written, after localization, as
\begin{equation*}
\begin{aligned}
   e^{-rt}G(X_t)
   &=
   G(x)
   +\int_0^t e^{-rs}
      \big(\mathcal L G_{\mathrm{sm}}(X_s)-rG(X_s)\big)
      \mathbf 1_{\{X_s\notin\Delta\}}\,ds
   \nonumber\\
   &\quad
   +\frac12\int_0^t e^{-rs}\,dL_s^0(Y)
   +M_t,
\end{aligned}
\end{equation*}
where $M$ is a local martingale and the reflection contribution has vanished.
Equivalently, using the signed measure $\Gamma$ in
\eqref{eq:Gamma-decomposition},
\begin{equation}
\label{eq:ito-tanaka-G-Gamma}
   e^{-rt}G(X_t)
   =
   G(x)
   +\int_0^t e^{-rs}c(X_s)\,ds
   -\int_0^t e^{-rs}\,dA_s^\Gamma
   +M_t,
\end{equation}
where $A^\Gamma$ is the signed continuous additive functional associated with
the signed measure $\Gamma$ \cite{revuz2013continuous,fukushima2011dirichlet}.
Identity \eqref{eq:ito-tanaka-G-Gamma} is the convention used in the
killed-resolvent representation of the value function.

\section{Killed-resolvent representation and conditional boundary
characterization}
\label{sec:killed-resolvent}
This section gives the potential-theoretic representation of the
value function. The relevant resolvent is the resolvent of the reflected
diffusion killed upon entry into the stopping set, not the unrestricted
reflected resolvent of the process continued after stopping. Throughout this
section we work under the assumptions of
Sections~\ref{sec:model-verification}--\ref{sec:singular-stopping-gain}. In
particular, $\Gamma=\Gamma^{\mathrm{ac}}\,dx+\Gamma^\Delta$ denotes the signed
stopping-gain measure defined in \eqref{eq:Gamma-decomposition}, and
$A^\Gamma$ denotes the signed continuous additive functional associated with
$\Gamma$.

\subsection{The killed resolvent}
\label{sec:killed-resolvent-definition}
Let $\tau_{\mathcal D}:=\inf\{t\ge0:X_t\in\mathcal D\}$ be the first entry time
into the stopping set. For a signed smooth measure $\nu$ for which the
following expression is well defined, define the resolvent killed upon entry
into $\mathcal D$ by
$\displaystyle    R_r^{\mathcal C}\nu(x)
   :=
   \mathbb E_x\left[
      \int_0^{\tau_{\mathcal D}} e^{-rs}\,dA_s^\nu
   \right]$,
$ x\in\R^2_+$.
If $\nu=f\,dx$, this becomes
\begin{equation}
\label{eq:killed-resolvent-function}
   R_r^{\mathcal C}f(x)
   =
   \mathbb E_x\left[
      \int_0^{\tau_{\mathcal D}} e^{-rs}f(X_s)\,ds
   \right].
\end{equation}
By contrast, the unrestricted reflected resolvent is
\begin{equation}
\label{eq:unrestricted-resolvent}
   R_r^{\mathrm R}f(x)
   :=
   \mathbb E_x\left[
      \int_0^\infty e^{-rs}f(X_s)\,ds
   \right].
\end{equation}
In general,
\begin{equation}
\label{eq:killed-not-unrestricted}
   R_r^{\mathcal C}f(x)
   \ne
   R_r^{\mathrm R}(f\mathbf 1_{\mathcal C})(x).
\end{equation}
Indeed, the unrestricted process in \eqref{eq:unrestricted-resolvent} continues
to evolve after $\tau_{\mathcal D}$ and may later re-enter $\mathcal C$. The
killed resolvent \eqref{eq:killed-resolvent-function} counts only the
occupation accumulated before the first entry into $\mathcal D$.

\begin{assumption}[Admissibility of killed potentials]
\label{ass:killed-potential-admissibility}
For every signed measure $\nu$ used below, the positive and negative parts
$\nu^+$ and $\nu^-$ are smooth measures for the reflected diffusion and satisfy
$R_r^{\mathcal C}\nu^+(x)+R_r^{\mathcal C}\nu^-(x)<\infty$ for $x\in\R^2_+$. The
killed potential $R_r^{\mathcal C}\nu$ is assumed to admit the
continuation-side traces used below.
\end{assumption}

\begin{remark}[Killing is essential]
\label{rem:killing-essential}
Formula \eqref{eq:killed-not-unrestricted} is the central object in this
section. The optimal stopping representation obtained from It\^o--Tanaka
applies only up to the optimal stopping time. Therefore the potential term is a
killed potential over the pre-stopping path. Replacing it by the full reflected
resolvent of $f\mathbf 1_{\mathcal C}$ would incorrectly include post-stopping
occupation of $\mathcal C$.
\end{remark}

\subsection{Value representation}
\label{sec:correct-value-representation}
The following result is the resolvent representation of the value
function.

\begin{theorem}[Killed-resolvent representation]
\label{thm:killed-resolvent-representation}
Assume that $\tau_{\mathcal D}$ is optimal and that the generalized
It\^o--Tanaka identity \eqref{eq:ito-tanaka-G-Gamma} may be applied up to
$\tau_{\mathcal D}$, with the required localization and uniform integrability.
Then, for every $x\in\R^2_+$,
\begin{equation}
\label{eq:value-killed-resolvent}
   V(x)=G(x)-R_r^{\mathcal C}\Gamma(x).
\end{equation}
Equivalently,
\[
   V(x)
   =
   G(x)
   -\mathbb E_x\left[
      \int_0^{\tau_{\mathcal D}}e^{-rs}\,dA_s^\Gamma
   \right].
\]
\end{theorem}
\begin{proof}
Since $\tau_{\mathcal D}$ is optimal,
$\displaystyle    V(x)
   =
   \mathbb E_x\left[
      e^{-r\tau_{\mathcal D}}G(X_{\tau_{\mathcal D}})
      -\int_0^{\tau_{\mathcal D}}e^{-rs}c(X_s)\,ds
   \right]$.
Apply the generalized It\^o--Tanaka identity \eqref{eq:ito-tanaka-G-Gamma} to
$G(X)$ on $[0,t\wedge\tau_{\mathcal D}]$:
\[
   e^{-r(t\wedge\tau_{\mathcal D})}G(X_{t\wedge\tau_{\mathcal D}})
   =
   G(x)
   +\int_0^{t\wedge\tau_{\mathcal D}}e^{-rs}c(X_s)\,ds
   -\int_0^{t\wedge\tau_{\mathcal D}}e^{-rs}\,dA_s^\Gamma
   +M_{t\wedge\tau_{\mathcal D}},
\]
where $M$ is a local martingale. After localization, expectation removes the
martingale term. Passing to the limit $t\to\infty$ using the assumed uniform
integrability gives
\[
   \mathbb E_x\left[
      e^{-r\tau_{\mathcal D}}G(X_{\tau_{\mathcal D}})
      -\int_0^{\tau_{\mathcal D}}e^{-rs}c(X_s)\,ds
   \right]
   =
   G(x)
   -\mathbb E_x\left[
      \int_0^{\tau_{\mathcal D}}e^{-rs}\,dA_s^\Gamma
   \right].
\]
The right-hand side is precisely $G(x)-R_r^{\mathcal C}\Gamma(x)$. This proves
\eqref{eq:value-killed-resolvent}.
\end{proof}

\begin{remark}[Incorrect unrestricted formula]
\label{rem:incorrect-unrestricted-formula}
The representation \eqref{eq:value-killed-resolvent} replaces the generally
incorrect identity
$V(x)=G(x)-R_r^{\mathrm R}(\Gamma\mathbf 1_{\mathcal C})(x)$. The latter would
be valid only under additional special circumstances preventing post-stopping
occupation of $\mathcal C$ by the unrestricted reflected process, or under a
convention in which the process is actually killed at $\tau_{\mathcal D}$.
\end{remark}

\subsection{Boundary trace condition}
\label{sec:boundary-trace-condition}
Assume now that the stopping set has the epigraph representation from
Theorem~\ref{thm:conditional-epigraph}:
$\mathcal D=\{(x_1,x_2)\in\R^2_+:x_2\ge b(x_1)\}$. Write
$\displaystyle    z_b(x_1):=(x_1,b(x_1))$.
Since $V=G$ on $\partial\mathcal C$, the killed potential in
\eqref{eq:value-killed-resolvent} has zero continuation-side trace on the free
boundary.

\begin{proposition}[Continuation-side trace condition]
\label{prop:boundary-trace}
Assume the hypotheses of Theorem~\ref{thm:killed-resolvent-representation} and
the free-boundary regularity needed for the following trace to exist. Then, for
every boundary point at which the trace is well-defined,
\begin{equation}
\label{eq:boundary-trace}
   \lim_{\substack{x\to z_b(x_1)\\ x\in\mathcal C}}
   R_r^{\mathcal C}\Gamma(x)
   =0.
\end{equation}
If non-tangential traces are used, \eqref{eq:boundary-trace} is interpreted as
$\lim_{x\to z_b(x_1),\,x\in\mathcal C\ \mathrm{n.t.}}R_r^{\mathcal C}\Gamma(x)=0$.
\end{proposition}
\begin{proof}
For $x\in\mathcal C$, Theorem~\ref{thm:killed-resolvent-representation} gives
$R_r^{\mathcal C}\Gamma(x)=G(x)-V(x)$. Let $x\to z_b(x_1)$ from within
$\mathcal C$. Since $V$ and $G$ are continuous and $V=G$ on the boundary, the
right-hand side tends to zero.
\end{proof}

\begin{remark}[Why the boundary value itself is not the equation]
\label{rem:boundary-value-trivial}
One should not write $R_r^{\mathcal C}\Gamma(z_b(x_1))=0$ as a substantive
boundary equation unless a nontrivial trace convention has been specified. If
the killed process is started at a point of $\mathcal D$, then
$\tau_{\mathcal D}=0$, and the killed potential vanishes trivially. The
meaningful condition is the continuation-side trace \eqref{eq:boundary-trace}.
\end{remark}

When a killed Green kernel exists, the killed potential may be written
schematically as
\begin{equation}
\label{eq:killed-green-representation}
   R_r^{\mathcal C}\Gamma(x)
   =
   \int_{\mathcal C}G_r^{\mathcal C}(x,y)
      \Gamma^{\mathrm{ac}}(y)\,dy
   +R_r^{\mathcal C}\Gamma^\Delta(x),
   \qquad x\in\mathcal C.
\end{equation}
Here $G_r^{\mathcal C}$ is the Green kernel of the reflected diffusion killed
upon entry into $\mathcal D$, and $R_r^{\mathcal C}\Gamma^\Delta$ is the killed
potential of the diagonal surface measure
$\Gamma^\Delta(dx)=-\frac{q(x)}{2\sqrt{1+\alpha^2}}\,\sigma_\Delta(dx)$. The
boundary condition is the continuation-side trace of
\eqref{eq:killed-green-representation}.

\subsection{Candidate verification rather than Fredholm uniqueness}
\label{sec:candidate-verification}
Let $h:\R_+\to\R_+$ be a candidate boundary. Define
$   \mathcal D_h:=\{(x_1,x_2)\in\R^2_+:x_2\ge h(x_1)\}$,
$   \mathcal C_h:=\R^2_+\setminus\mathcal D_h$,
and
$   \tau_h:=\inf\{t\ge0:X_t\in\mathcal D_h\}$.
For an admissible signed measure $\nu$, define the killed resolvent relative to
$\mathcal C_h$ by
$\displaystyle    R_r^{\mathcal C_h}\nu(x)
   :=
   \mathbb E_x\left[
      \int_0^{\tau_h}e^{-rs}\,dA_s^\nu
   \right]$.
The candidate value associated with $h$ is
\begin{equation}
\label{eq:Uh-def}
   U_h(x):=G(x)-R_r^{\mathcal C_h}\Gamma(x).
\end{equation}
The next theorem is a conditional verification result. It is not a uniqueness
theorem for a Fredholm equation.

\begin{theorem}[Candidate verification]
\label{thm:candidate-verification}
Let $h$ be a candidate boundary and define $U_h$ by \eqref{eq:Uh-def}. Assume:
\begin{enumerate}
\item[\textup{(C1)}] $U_h$ is continuous and of polynomial growth.
\item[\textup{(C2)}] $U_h\ge G$ on $\R^2_+$.
\item[\textup{(C3)}] $U_h=G$ on $\mathcal D_h$.
\item[\textup{(C4)}] In $\mathcal C_h$, $(\mathcal L-r)U_h=c$ in the weak,
viscosity, or generalized It\^o sense required by
Theorem~\ref{thm:verification-measure}.
\item[\textup{(C5)}] $U_h$ satisfies the reflected Neumann condition on the
coordinate axes in the trace or reflected-viscosity sense.
\item[\textup{(C6)}] $U_h$ is globally measure-superharmonic relative to the
running cost: $(\mathcal L-r)U_h-c\le0$ as a signed measure in the interior, in
the sense required by Theorem~\ref{thm:verification-measure}.
\item[\textup{(C7)}] The localization, additive-functional, and
uniform-integrability hypotheses of Theorem~\ref{thm:verification-measure} hold
for $U_h$ and $\tau_h$.
\end{enumerate}
Then $U_h=V$ on $\R^2_+$, and $\tau_h$ is optimal. If, in addition, the true
stopping set is known to be an epigraph
$\mathcal D=\{(x_1,x_2):x_2\ge b(x_1)\}$, then $h=b$ at every point where both
graph representatives are defined with the same convention.
\end{theorem}
\begin{proof}
By \textup{(C1)}, \textup{(C2)}, \textup{(C4)}, \textup{(C5)}, \textup{(C6)},
and \textup{(C7)}, the verification theorem
(Theorem~\ref{thm:verification-measure}) applies to $U_h$ with
$\mathcal O=\mathcal C_h$. Therefore $U_h\ge V$. Applying the stopped
generalized It\^o formula to $U_h(X)$ up to $\tau_h$, using
$(\mathcal L-r)U_h=c$ in $\mathcal C_h$, the reflected boundary condition, and
the integrability hypotheses, gives
\[
   U_h(x)
   =
   \mathbb E_x\left[
      e^{-r\tau_h}U_h(X_{\tau_h})
      -\int_0^{\tau_h}e^{-rs}c(X_s)\,ds
   \right].
\]
By the contact condition \textup{(C3)}, $U_h(X_{\tau_h})=G(X_{\tau_h})$ on
$\{\tau_h<\infty\}$. Hence
\[
   U_h(x)
   =
   \mathbb E_x\left[
      e^{-r\tau_h}G(X_{\tau_h})
      -\int_0^{\tau_h}e^{-rs}c(X_s)\,ds
   \right]
   \le V(x).
\]
Thus $U_h=V$, and $\tau_h$ is optimal. If the true stopping set is known to be
an epigraph with boundary $b$, then $\{U_h=G\}=\{V=G\}=\mathcal D$. Since
$\mathcal D_h$ is also an epigraph and $U_h=G$ on $\mathcal D_h$ by
\textup{(C3)}, the two epigraphs coincide under the stated convention, and
therefore $h=b$.
\end{proof}

\begin{remark}[No uniqueness from the trace condition alone]
\label{rem:no-uniqueness-from-trace}
The boundary trace condition
\[
   \lim_{\substack{x\to z_h(x_1)\\ x\in\mathcal C_h}}R_r^{\mathcal C_h}\Gamma(x)=0
\]
does not
by itself imply uniqueness of $h$. It also does not imply $U_h=G$ throughout
$\mathcal D_h$. Uniqueness of a candidate boundary follows only after a
verification argument: one must prove majorization $U_h\ge G$, contact on the
whole candidate stopping set, global measure-superharmonicity, the reflected
boundary condition, and the required integrability properties.
\end{remark}

\begin{remark}[Main conclusion]
\label{rem:main-corrected-conclusion}
The defensible structural conclusion is therefore:
\[
   \begin{minipage}{0.86\linewidth}
   \centering
   Given structural monotonicity and regularity assumptions, the reflected
   optimal stopping problem admits a killed-resolvent boundary-trace
   representation.
   \end{minipage}
\]
The theory does not assert unsupported general free-boundary regularity, an
unrestricted-resolvent Fredholm equation, or an epigraph theorem derived from
order preservation alone.
\end{remark}


\end{document}